\documentclass[reqno]{amsart}
\usepackage[utf8]{inputenc}
\usepackage[T1]{fontenc}
\usepackage{enumerate}
\usepackage{lmodern}

\usepackage{color, ulem}

\definecolor{darkgreen}{rgb}{0,0.6,0}
\definecolor{darkred}{rgb}{0.7,0,0}
\definecolor{darkblue}{rgb}{0,.2,.7}

 \usepackage[%
  colorlinks, 
  citecolor=darkgreen, linkcolor=darkblue,
   pdfpagelabels=true,
  unicode=true,
  pdfusetitle
 ]{hyperref}

\allowdisplaybreaks

\theoremstyle{plain}
\makeatletter
\newtheorem*{rep@theorem}{\rep@title}
\newcommand{\newreptheorem}[2]{%
\newenvironment{rep#1}[1]{%
 \def\rep@title{#2 \ref{##1}}%
 \begin{rep@theorem}}%
 {\end{rep@theorem}}}
 
\makeatother

\newtheorem{thm}{Theorem}[section] 
 
\newtheorem{cor}[thm]{Corollary}
\newtheorem{corollary}[thm]{Corollary} 
 
\newtheorem{lemma}[thm]{Lemma}

\theoremstyle{definition}

\newtheorem{example}[thm]{Example}

\renewcommand{\epsilon}{\varepsilon}

\let\theta\vartheta
\let\phi\varphi

\let\<\langle
\let\>\rangle

\DeclareMathAlphabet{\doba}{U}{msb}{m}{n}
\gdef\i{\mathrm{i}}
\gdef\d{\mathrm{d}}
\gdef\mC{\doba{C}}

 \def\scal{{\mathop{\rm scal}}}
 \def\Ric{{\mathop{\rm Ric}}}
\def\C{\mathbb C}
\def\Spin{{\mathop{\rm Spin}}}
\def\spin{{\mathop{\rm Spin}}}

\newcommand{\Spinc}{\mathrm{Spin}^c}
\newcommand{\trace}{\mathrm{tr\,}}

\newcommand{\define}{\mathrel{:=}}
\newcommand{\ccdot}{\!\cdot\!}

\begin{document}
\title[Totally Umbilic Hypersurfaces of  \texorpdfstring{$\Spinc$}{Spinc} manifolds with special  spinors]{Totally Umbilical Hypersurfaces of  \texorpdfstring{$\Spinc$}{Spinc} manifolds carrying special  spinor fields}

\author{Nadine Grosse}
 \address{Mathematical Institute \\ 
 Universit\"at Freiburg \\
 79104 Freiburg \\
 Germany.}
\email{nadine.grosse@math.uni-freiburg.de}

\author{Roger Nakad}
\address{Department of Mathematics and Statistics\\ 
Faculty of Natural and Applied Sciences \\
Notre Dame University-Louaize\\
 P.O. Box 72, Zouk Mikael, Zouk Mosbeh\\
Lebanon.}
\email{rnakad@ndu.edu.lb}



\thanks{Acknowledgement: This work was initiated during the stay of both authors  at the "Centre International de Rencontres Math\'ematiques" in Luminy, Marseille-France (Research in Pairs program). Then, it was continued during the stay of both authors at CIRM (Centro Internazionale per la Ricerca Matematica) in Italy (Research in Pairs program). The authors gratefully acknowledge the support of both centers. The second author thanks the Institute of Mathematics of the University of Freiburg for its continuous support and hospitality during many research stays.
The first author was supported by the Juniorprofessurenprogramm Baden-Württemberg}

 \begin{abstract} 
Under some dimension restrictions, we prove that totally umbilical  hypersurfaces of $\Spinc$ manifolds carrying a parallel, real or imaginary Killing spinor are of constant mean curvature. This extends to the $\Spinc$ case the result of O. Kowalski stating that, every totally umbilical  hypersurface of an Einstein manifold of dimension greater or equal to $3$ is of constant mean curvature. As an application, we
prove that there are no extrinsic hypersheres in complete Riemannian  $\Spin$ manifolds of
non-constant sectional curvature carrying a parallel, Killing or imaginary Killing spinor.
\end{abstract}

 \maketitle

\section{Introduction}

Using  classical submanifold techniques, a lot of results on the geometry  
of totally umbilical   submanifolds (and other special hypersurfaces) in ambient manifolds of special geometries were obtained \cite{BCO, Chen1, chen1, chen2, chen3, chen4, ok, to, tsu, yam}. As one example, O. Kowalski \cite{Kow} used the Codazzi-Mainardi equation to prove  the following  elementary and well-known result:
\begin{thm} \label{bfact}Every totally umbilical connected hypersurface of an Einstein manifold of dimension greater or equal to 3 is of constant mean curvature.
\end{thm}

Examples of ambient Riemannian Einstein manifolds $(\widetilde M^{m+1}, \widetilde g)$ of dimension 
 $m+1 \geq 3$ are Riemannian $\Spin$ manifolds carrying an $\alpha$-Killing spinor ($\alpha\in \mathbb C$), i.e., a spinor field $\psi$ 
satisfying the equation 
\begin{equation}\label{def:Killing} \widetilde \nabla_X \psi = \alpha X\cdot \psi,\end{equation}
for any  vector $X$ tangent to $\widetilde M$, where $\widetilde \nabla$ denotes the spinorial Levi-Civita connection on the spinor bundle and $``\cdot"$ the Clifford multiplication, compare Section~\ref{prelim}.\medskip

For $\Spin$ manifolds it is known that the Killing constant $\alpha$ has to be zero (parallel spinor), a nonzero real constant (real Killing spinor) or a nonzero purely imaginary constant (imaginary Killing spinor) \cite{BHMM}. When $\alpha$ is real, such spinors characterize the limiting case in the Friedrich's and Hijazi's inequalities which provide a lower bound for the eigenvalues of the Dirac operator involving the infimum of the scalar curvature or the first eigenvalue of the Yamabe operator \cite{frfr, hij84, hij86}. Moreover, the existence of $\alpha$-Killing spinors leads to restrictions on the geometry and topology of the manifold. In fact besides being Einstein (and even Ricci-flat when $\alpha =0$), $\widetilde M$ is automatically compact if $\alpha$ is real and noncompact if $\alpha$ is purely imaginary.  Complete simply connected $\Spin$ manifolds with real, parallel or imaginary Killing  spinors  have been classified by Wang \cite{wang}, B{\"a}r \cite{bar1} and Baum  \cite{baum1, baum2, baum3} and the existence is glued to the holonomy of the manifold. This classification gives, in some dimensions, other examples than the most symmetric ones as Euclidean space, the sphere  or the hyperbolic space. These examples are relevant to physicists in general relativity where the Dirac operator plays a central role. \medskip
 
Techniques from $\Spin$ geometry  have been successfully used to produce striking advances in extrinsic geometry (see e.g. the study of CMC or minimal surfaces in homogeneous 3-spaces which arise in Thurston's classification of 3-dimensional geometries and Alexandrov-type theorems as in \cite{HM13, HS14, HMZ01, HMR03, Bar98}). It is remarkable that, in many extrinsic results, $\Spin$ geometrical tools - in particular special/natural spinor fields and  the Dirac operator - have played a central role and  inspired further research directions.\medskip

When shifting from the from the classical $\spin$ geometry to $\Spinc$ geometry, the situation is more general and many obstacles appear since the $\Spinc$ structure will not only depend on the geometry of the manifold but also on the connection (and hence the curvature) of the auxiliary line bundle associated with the fixed $\Spinc$ structure. From a physical point of view, spinors model fermions while $\Spinc$-spinors can be interpreted as fermions coupled to an electromagnetic field. Transferring the idea to use spinorial methods in the study of submanifolds to the $\Spinc$ world, allows us to cover more ambient geometric structures (CR structures, K\"ahler and Sasaki structures). Indeed, O. Hijazi, S. Montiel and F. Urbano constructed on K\"ahler-Einstein  manifolds  with  positive  scalar  curvature \cite{HMU}, $\Spinc$
structures  carrying K\"ahlerian  Killing  spinors.  The  restriction  of  these  spinors  to  minimal  Lagrangian
submanifolds provides topological and geometric restrictions on these submanifolds (see \cite{RN12, Na11} for other applications of the use of $\Spinc$ geometry in  extrinsic geometry). Equation \eqref{def:Killing} on $\Spinc$ manifolds has been studied  by A. Moroianu \cite{moroi}  when $\alpha$ is real and by the authors \cite{grosse-nakad} when $\alpha$ is purely imaginary. In fact, a complete simply connected manifold has a parallel $\Spinc$ spinor if and only if it is isometric to the Riemannian product between a simply connected K\"ahler manifold (with its canonical or anti-canonical $\Spinc$ structure) and a simply connected $\Spin$ manifold carrying a parallel spinor. The only simply connected $\Spinc$ manifolds admitting real non-parallel Killing spinors other than the $\Spin$ manifolds are the non-Einstein Sasakian manifolds
endowed with their canonical or anti-canonical $\Spinc$ structure. Beside that, complete $\Spinc$ manifolds
with imaginary Killing spinors are isometric to a special warped product of a $\Spinc$ manifold with a parallel spinor with  $\mathbb R$. These classification results, stated above, show that, in contrast to the $\Spin$ case,  $\Spinc$ manifolds carrying an $\alpha$-Killing spinor are not in general Einstein manifolds. For this reason, we start by extending  Theorem~\ref{bfact} to ambient Riemannian $\Spinc$ manifolds carrying an $\alpha$-Killing spinor. In fact, our main result is:
\begin{thm}\label{Exto}
Every connected totally umbilical  hypersurface $M^m$  of a Riemannian $\Spinc$ manifold $\widetilde M^{m+1}$ with $m+1\geq 5$  carrying a real (including parallel) Killing spinor or of dimension $m+1\geq 3$ carrying an imaginary Killing spinor is of constant mean curvature. \end{thm}
Here, we recall that if $\widetilde{M}$ is already $\Spin$, Theorem~\ref{Exto} is just a special case of Theorem~\ref{bfact} because in this case $\widetilde M$ is Einstein. Using the classification of Killing spinors on $\Spinc$ manifolds cited above we thus in particular obtained:

\begin{corollary}
 Every connected totally umbilical  hypersurface $M^m$  of a Kähler or a Sasakian manifold has constant mean curvature.
\end{corollary}

For Kähler manifolds the last statement is known from \cite[Thm. 4.2]{Chen1} and we merely give a spinorial proof here. The counterpart for Sasakian manifolds was not known before to our best knowledge.\medskip

We will show by counterexamples that  Theorem~\ref{Exto} is sharp in the sense that it fails if the ambient $\Spinc$ manifold is of dimension  $3$ or $4$  carrying a parallel or  real Killing spinor.  The proof of Theorem~\ref{Exto} relies on two  families of differential forms naturally associated to the spinor  obtained by the restriction of the $\alpha$-Killing spinor to the hypersurface $M$. These differential forms and their exterior derivatives  involve the mean curvature $H$ of the isometric immersion and hence allow to deduce that $H$ is constant. Dependent on whether $\alpha$ is real or imaginary, the proof of Theorem~\ref{Exto} differs in these cases and is carried out separately (see Section~\ref{Mmain}). \medskip
 
As further applications of Theorem~\ref{Exto}, we give some no-existence results of extrinsic hyperspheres in some special complete  $\Spin$ manifolds.
\begin{thm}\label{THMp1}
There are no extrinsic hyperspheres in 
\begin{enumerate}[(i)]
\item complete manifolds with holonomy $G_2$ and Spin$(7)$.
%\item complete Hyperkähler and Calabi-Yau manifolds of nonconstant sectional curvature
 \item complete simply connected 3-Sasakian manifold of dimension $4m+3$  which is not of constant curvature
\item complete simply connected Sasakian Einstein manifold of dimension $4m+3$, $m \geq 2$ which is not 3-Sasakian 
\item compact Sasakian-Einstein manifold of dimension $2m+1$, $m \geq 2$ which are not  locally symmetric.
\item homogeneous warped product  $\widetilde M = N \times_f  \mathbb R$ where $f(t) = e^{4 \mu t}$ ($\mu \in \mathbb R^*$) and $N$ a complete Riemannian $\spin$ manifold with a parallel spinor and of non-constant sectional curvature.
\end{enumerate}
\end{thm}

 Note that (i) was already obtained in \cite{JMS} and we give here just a spinorial proof. Moreover, Theorem~\ref{THMp1} is a particular case of the more general Theorem~\ref{pa-ap}. In fact, we prove that there are no extrinsic hyperspheres in Riemannian $\Spin$ manifolds of non-constant sectional curvature and carrying an $\alpha$-Killing spinor field.

 %%%%%%%%%%%%%%%%%%%%%%%%%%%%%%%%%%%%%%%%%%%%%%%%%%%%%%%%%%%%%%%%%%%%
\section{Preliminaries}\label{prelim}
%%%%%%%%%%%%%%%%%%%%%%%%%%%%%%%%%%%%%%%%%%%%%%%%%%%%%%%%%%%%%%%%%%%%

In this section, we briefly review 
some basic facts about $\Spinc$ structures on oriented Riemannian manifolds and their hypersurfaces \cite{friedrich, LM, BHMM, Bar98, Na11, bfg}.

\subsection{Hypersurfaces and induced \texorpdfstring{$\Spinc$}{Spinc} structures}\hfill \medskip

\textbf{Spin$^c$ structures on manifolds:} Let $(\widetilde M^{m+1}, \widetilde g)$ be a Riemannian $\Spinc$ manifold of dimension $m+1 \geq 3$ without
boundary. On such a manifold, we have  a Hermitian complex vector bundle $\Sigma \widetilde M$ endowed with a natural scalar product $\<., .\>$ and with a connection $\widetilde \nabla $ which parallelizes the metric. We denote by $\Re\<.,.\>$ the real part of the scalar product $\<., .\>$. This complex vector bundle, called the $\Spinc$ bundle, is endowed with a Clifford multiplication denoted by $``\cdot"$, $\cdot\colon T\widetilde M \rightarrow \mathrm{End}_{\mathbb C} (\Sigma \widetilde M)$, such that at every point $x \in \widetilde M$, defines an irreducible representation of the corresponding Clifford algebra. Hence, the complex rank of $\Sigma \widetilde M$ is $2^{[\frac {m+1}{2}]}$. The Clifford multiplication can be extended to exterior products of the tangent bundle and to differential forms, such that $(v_1\wedge \ldots \wedge v_k)\cdot \varphi\define v_1\cdot \ldots \cdot v_k\cdot \varphi$ if the $v_i$'s are mutually orthogonal and such that $v^\sharp\cdot \psi\define  v\cdot \psi$ for all vector fields $v_i,v$ and spinors $\psi$ and where $.^\sharp$ denotes the isomorphism $TM\to T^*M$ induced by the metric.\medskip

Given a $\Spinc$ structure on $(\widetilde M^{m+1}, g)$, one can prove that the determinant line bundle $\mathrm{det} (\Sigma \widetilde M)$ has a root of index $2^{[\frac{m+1}{2}]-1}$. We denote
by $\widetilde L$ this root line bundle over $\widetilde M$ and it is called the auxiliary line bundle associated with the $\Spinc$ structure. Locally, a $\Spin$ structure always exists. We denote by $\Sigma' \widetilde M$ the (possibly globally non-existent)
spinor bundle. Moreover, the square root of the auxiliary line bundle $\widetilde L$
always exists locally. But, $\Sigma\widetilde M = \Sigma' \widetilde M \otimes {\widetilde L}^{\frac 12}$ exists globally.  This essentially means that, while
the spinor bundle and ${\widetilde L}^{\frac 12}$
may not exist globally, their tensor product (the $\Spinc$  bundle) is
defined globally. Thus, the connection $\widetilde \nabla$ on $\Sigma \widetilde M$ is the twisted connection of the one on the
spinor bundle (coming from the Levi-Civita connection) and a fixed connection on $\widetilde L$. \medskip 

We may now define the Dirac operator $\widetilde D$ acting on the space of smooth
sections of $\Sigma\widetilde M$  by the composition of the metric connection and
the Clifford multiplication. In local coordinates this reads as
\begin{align*}\widetilde D =\sum_{j=1}^{m+1} e_j \cdot \widetilde \nabla_{e_j},
 \end{align*}
where $\{e_1,\ldots,e_{m+1}\}$ is a local oriented orthonormal tangent frame. It is a first order elliptic operator, formally self-adjoint with respect to the $L^2$-scalar product and satisfies the Schr\"odinger-Lichnerowicz formula 
\begin{align}\label{eq:Lich}
{\widetilde D}^2=\widetilde\nabla^*\widetilde \nabla+\frac{1}{4}\widetilde {\mathrm{scal}}+\frac{\i}{2}\widetilde\Omega\cdot,
\end{align}
where $\widetilde\nabla^*$ is the adjoint of $\widetilde\nabla$
with respect to the $L^2$-scalar product, $\widetilde {\mathrm{scal}}$ is the scalar curvature of $\widetilde M$, $\i\widetilde \Omega$ is the curvature of
the auxiliary line bundle $\widetilde L$ associated with the fixed connection ($\widetilde \Omega$ is a real $2$-form on $\widetilde M$) and $\widetilde \Omega \cdot$
is the extension of the Clifford multiplication to differential forms. For any $X \in \Gamma(T\widetilde M)$ and any spinor field $\psi \in \Gamma (\Sigma \widetilde M)$ , the Ricci identity is given by
\begin{align}\label{Ricci-identity}
\sum_{k=1}^{m+1} e_k \ccdot \widetilde{\mathcal{R}}(e_k,X) \psi=
\frac 12 \widetilde \Ric(X) \ccdot \psi-\frac{\i}{2} (X\lrcorner\widetilde\Omega)\cdot\psi,
\end{align}
where $\widetilde{\Ric}$ is the Ricci curvature of $(\widetilde M^{m+1}, g)$ and $\widetilde {\mathcal{R}}$ is the curvature tensor of the spinorial connection $\widetilde\nabla$. \medskip

When $m$ is even, the complex volume form $\widetilde\omega_{\C} \define \i^{[\frac{m+2}{2}]} e_1 \cdot\ldots \cdot e_{m+1}$ acts on $\Sigma \widetilde M$ as the identity, i.e., $\widetilde\omega_\C \ccdot\psi = \psi$ for any spinor $\psi \in \Gamma(\Sigma \widetilde M)$. Besides, if $m$ is odd, we have $\widetilde\omega_\C^2 =1$. We denote by $\Sigma^\pm \widetilde M$ the eigenbundles corresponding to the eigenvalues $\pm 1$, hence $\Sigma\widetilde  M = \Sigma^+ \widetilde M \oplus \Sigma^- \widetilde M$ and a  spinor field $\psi$ can be written as $\psi = \psi^+ + \psi^-$. The conjugate $\overline \psi$ of $\psi$ is defined  by $\overline \psi = \psi^+ - \psi^-$.\medskip

 As pointed out, the Clifford multiplication can be extended to differential forms and one sees that
\begin{eqnarray*}
\<\delta \cdot \psi, \psi\>= (-1)^{\frac{k(k+1)}{2}}\  \overline{\<\delta\cdot\psi, \psi\>}
\end{eqnarray*}
for any $k$-form $\delta$ and a spinor field $\psi\in \Gamma (\Sigma \widetilde M)$. 
This directly implies that for  mutually orthogonal vector fields $v_{1},\ldots, v_k$ we have
\begin{align}\label{cli-diff}
\<v_1\cdot v_2\cdot \ldots \cdot v_k \cdot \psi, \psi\>\in \left\{\begin{matrix}
                                                                   \mathbb R & \text{for }k\equiv 0,3\ \text{ mod  } 4\\
                                                                   \i \mathbb R & \text{for }k\equiv 1,2\ \text{ mod  } 4.                                                                                                                                     \end{matrix}
 \right.
\end{align}

\textbf{Spin$^c$ structures on  hypersurfaces:} The following can be e.g. found in \cite{r2}.  Any $\Spinc$ structure on  $ (\widetilde M^{m+1}, g)$ induces a $\Spinc$ structure on an oriented hypersurface $(M^m, g)$ of dimension $m\geq 2$, and we have 
$$ \Sigma M\simeq \left\{
\begin{array}{l}
\Sigma {\widetilde M}_{|_M} \ \ \ \ \ \ \text{\ \ \ if\ $m$ is even,} \\\\
 \Sigma^+ {\widetilde M}_{|_M}   \ \text{\ \ \ \ \ \ if\ $m$ is odd.}
\end{array}
\right.
$$
Furthermore Clifford multiplication by a vector field $X$, tangent to $M$, is given by 
\begin{align*}
X\ccdot_M\phi = (X\ccdot\nu\ccdot \psi)_{|_M},
\end{align*}
where $\psi \in  \Gamma(\Sigma \widetilde M)$ (or $\psi \in \Gamma(\Sigma^+ \widetilde M)$ if $m$ is odd),
$\phi$ is the restriction of $\psi$ to $M$, ``$\ccdot_M$'' the Clifford multiplication  on $M$ and $\nu$ is the unit normal
vector field of $M$ in $\widetilde M$. Also, when $m$ is odd, we obtain $\Sigma M \simeq \Sigma^- \widetilde M \vert_M$. With this identification, the Clifford multiplication is given by
$X\ccdot_M\phi = -(X\ccdot\nu\ccdot \psi)_{|_M}$. In particular, we have $\Sigma \widetilde M \simeq \Sigma M \oplus \Sigma M$.\medskip

Moreover, the corresponding auxiliary line bundle $L$ on $M$ is the restriction to $M$ of $\widetilde L$ and the curvature $2$-form $\i\Omega $ on $L$ is given by  $\i\Omega = \i \widetilde \Omega\vert_{M}$.  For every
$\psi \in \Gamma(\Sigma \widetilde M)$ ($\psi \in \Gamma(\Sigma^+ \widetilde M)$ if $m$ is odd), the real 2-forms
$\Omega$ and $\widetilde \Omega$ are related by
\begin{align*}
(\widetilde\Omega \ccdot\psi)_{|_M} = \Omega\ccdot_M\phi -
(\nu\lrcorner\widetilde\Omega)\ccdot_M\phi.
\end{align*}
We denote by $\nabla$ the $\Spinc$  connection on $\Sigma M$. Then, for all $X\in \Gamma(TM)$, we have the $\Spinc$ Gauss formula:
\begin{equation}\label{eq_spincgauss}
(\widetilde\nabla_X\psi)_{|_M} =  \nabla_X \phi + \frac 12 \mathrm{II} X \ccdot_M\phi,
\end{equation}
where $\mathrm{II}$ denotes the Weingarten map of the hypersurface. Denoting by $D$  the Dirac operator on $M$ and  by the same symbol any spinor and its restriction to $M$, we have
\begin{equation*}
D \phi = \frac{m}{2}H\phi -\nu\ccdot \widetilde D\phi-\widetilde \nabla_{\nu}\phi,
\end{equation*}
where $H = \frac 1m \trace(\mathrm{II})$ denotes the mean curvature and $D^M = D$ if $m$ is even and $D^M= D\oplus(-D)$ if $m$ is odd.

%%%%%%%%%%%%%%%%%%%%%%%%%%%%%%%%%%%%%%%%%%%%%%%%%%%%%%%%%%%%%%%%%%%% 
\section{Totally umbilical  hypersurfaces of \texorpdfstring{$\Spinc$}{Spinc} manifolds carrying an \texorpdfstring{$\alpha$}{alpha}-Killing spinor}
%%%%%%%%%%%%%%%%%%%%%%%%%%%%%%%%%%%%%%%%%%%%%%%%%%%%%%%%%%%%%%%%%%%%
Let $(\widetilde{M}^{m+1}, \widetilde g)$ be a  Riemannian $\Spinc$ manifold with an $\alpha$-Killing spinor $\psi$ of Killing constant $\alpha\in \mC$. It is known that for $m\geq 1$, the Killing constant $\alpha$ has to be purely real or purely imaginary \cite[Theorem 1.1] {grosse-nakad}. Moreover, if $\alpha$ is real, then $\psi$ has constant norm since, for any $X \in \Gamma(T\widetilde M)$, we have
\begin{align}\label{eq:realconst} X(|\psi|^2)=2\Re \< \widetilde \nabla_X\psi, \psi\>=2\alpha \Re \< X\cdot \psi, \psi\>=0.\end{align}
Hence, real Killing spinors have no zeros. When $\alpha$ is purely imaginary, the function $\vert \psi\vert$ is a non-constant and
nowhere vanishing function \cite{baum1, grosse-nakad}. In this case, the set of zeros of $\psi$ is discrete \cite{grosse-nakad, rad, Lich1, KR}. 
Using the definition~\eqref{def:Killing} of an $\alpha$-Killing spinor $\psi$, we have 
\begin{align*}\widetilde{D}\psi = \sum_{j=1}^{m+1} e_j\ccdot \widetilde{\nabla}_{e_j} \psi= (m+1)\alpha \psi\quad \text{and}\quad
              \widetilde{D}^2\psi= (m+1)^2\alpha^2 \psi.
                                                       \end{align*}
Then, the Schr\"odinger-Lichnerowicz formula \eqref{eq:Lich} on $\widetilde{M}$ gives 
\begin{align*}
 m(m+1)\alpha^2\psi= \frac{\widetilde{\text{scal}}}{4}\psi+\frac{\i}{2} \widetilde{\Omega} \ccdot \psi.
\end{align*}

\emph{From now on we assume that $(M, g)$ is an oriented totally umbilical  hypersurface of $(\widetilde M, \widetilde g)$.} Totally umbilical  means $\mathrm{II}X=HX$ for all $X\in \Gamma(TM)$. Note that this implies $(\nabla_Y\mathrm{II})(X)=dH(Y)X$ for all $X,Y \in \Gamma(TM)$.\medskip 

We choose the local  orthonormal frame $e_i$ on $\widetilde{M}$ such that 
 $\{e_1 , \ldots, e_m\}$ is a local  orthonormal frame of $M$, $\nabla e_i =0$ and that $e_{m+1}=\nu$ a unit normal vector to $M$. The Ricci identity \eqref{Ricci-identity} on $\widetilde M$ for $X = \nu$ applied to the $\alpha$-Killing spinor $\psi$ reads
\begin{align}\label{ri-equation}  \frac{\widetilde{\Ric}(\nu,\nu)}{2}\nu\ccdot\psi+ \frac 12  \sum_{j=1}^m \widetilde {\mathrm{Ric}}(\nu, e_j) e_j \ccdot \psi - \frac{\i}{2} (\nu\lrcorner \widetilde \Omega)\ccdot \psi = 2m \alpha^2 \nu \ccdot \psi.\end{align}
where we also used the calculation $ \widetilde{R}(e_k,\nu)\psi=2\alpha^2 e_k\cdot \nu\cdot \psi$.
Now, the Codazzi-Mainardi equation \cite[Prop. 33]{ON}  gives that 
\begin{align*}
\widetilde g(\widetilde R (X, Y)U, \nu)&= g(\nabla_X \mathrm{II})(Y), U) - g((\nabla_Y \mathrm{II})(X), U)  \\
&= dH(X) \<Y, U\> - dH(Y) \<X, U\>% - \<\mathrm{II} (\nabla_XY  -\nabla_YX ),U\>
\end{align*}
for all $X, Y, U \in \Gamma(TM)$. Hence,  
$$\widetilde {\mathrm{Ric}} (X, \nu) =\sum_{l =1}^m \widetilde g(\widetilde R (X, e_l) e_l, \nu)=
(m-1) dH(X).$$
Replacing this in Equation \eqref{ri-equation} and taking then the Clifford multiplication by $\nu$, we obtain for $\phi=\psi|_M$ that \begin{align} \label{riccianddH}
-\frac 12 \widetilde \Ric(\nu, \nu) \varphi - \frac {(m-1)}{2} \d H\ccdot_M \varphi   +\frac{\i}{2} (\nu\lrcorner \widetilde \Omega)\ccdot_M \varphi = -2m \alpha^2 \varphi.
\end{align}
\begin{lemma}
Assume that $(M, g)$ is a totally umbilical  oriented hypersurface of $(\widetilde M, \widetilde g)$ carrying an $\alpha$-Killing spinor $\psi$. Then, for $\phi=\psi|_M$  we have
\begin{enumerate}
\item{{\bf The Schr\"odinger-Lichnerowicz formula on $M$:}}
\begin{align}\label{eq:LichM}
 m(m-1)\left( \alpha^2+\frac{ H^2}{4}\right)\phi +\frac{m-1}{2} dH\ccdot_M \phi = \frac{\text{scal}}{4} \phi +\frac{\i}{2} \Omega^M\ccdot_M\phi. 
\end{align}
\item{{\bf The Ricci identity on $M$:}}
\begin{align}\nonumber
\frac 12 \big(\mathrm{Ric} (X) - \i (X \lrcorner\Omega^M)\big)\ccdot_M \varphi =& -\frac 12 dH\ccdot_M X\ccdot_M \varphi - \frac m2 dH(X) \varphi \\ & + \frac{(m-1)}{2} H^2 X\ccdot_M\varphi +2(m-1) \alpha^2 X\ccdot_M \varphi. \label{Ricci-hyper}
\end{align}
\end{enumerate}
\end{lemma}
\begin{proof}
Using Equation \eqref{eq_spincgauss} we have:
\begin{align}
 \nabla_{e_i}\phi&=\widetilde{\nabla}_{e_i}\phi- \frac{H}{2}e_i\ccdot_M \phi= \alpha {e_i}\ccdot \phi-\frac{H}{2}e_i\ccdot_M \phi\label{eq_spincgaussKill}
 \end{align}
 and 
 \begin{align}
 \nabla_{e_j}\nabla_{e_i}\phi&=\nabla_{e_j}\!\left( \alpha {e_i}\ccdot \phi-\frac{H}{2}e_i\ccdot_M \phi\right)
\! =\! \alpha \nabla_{e_j} ({e_i}\ccdot \phi)-\frac{dH(e_j)}{2}e_i\ccdot_M \phi-\frac{H}{2} e_i\ccdot_M \nabla_{e_j}\phi\nonumber\\
 &=\alpha \widetilde{\nabla}_{e_j} ({e_i}\ccdot \phi)-\frac{\alpha H}{2} e_j\ccdot \nu\ccdot e_i\ccdot \phi-\frac{dH(e_j)}{2}e_i\ccdot_M \phi \nonumber \\
 & \quad -\frac{H}{2} e_i\ccdot \nu \ccdot \left(\alpha {e_j}\ccdot \phi-\frac{H}{2}e_j\ccdot \nu\ccdot \phi\right)\nonumber\\
 &=\alpha^2 e_i\ccdot e_j \ccdot  \phi + \alpha H \delta_{ij} \nu\ccdot \phi-\frac{\alpha H}{2} e_j\ccdot \nu\ccdot e_i\ccdot \phi-\frac{dH(e_j)}{2}e_i\ccdot_M \phi\nonumber\\ 
 &\quad -\frac{H}{2} e_i\ccdot \nu \ccdot \left(\alpha {e_j}\ccdot \phi-\frac{H}{2}e_j\ccdot \nu\ccdot \phi\right)\nonumber\\
 &=\left( \alpha^2 +\frac{H^2}{4}\right) e_i\ccdot e_j \ccdot  \phi -\frac{dH(e_j)}{2}e_i\ccdot_M \phi.\label{eq:nn}
\end{align}
Now, we calculate 
\begin{align*}
-\nabla^*\nabla \phi =&\sum_{i=1}^m \nabla_{e_i} \nabla_{e_i}\phi
    = -m \left( \alpha^2 +\frac{H^2}{4}\right)   \phi -\frac{1}{2}d H \ccdot_M \phi, 
\end{align*}
and for the Dirac operator on the hypersurface we obtain 
\begin{align}
 D\phi&=  \sum_{i=1}^m e_i\ccdot_M\!\left( \alpha {e_i}\ccdot \phi-\frac{H}{2}e_i\ccdot_M \phi \right) = m \alpha \nu\ccdot \phi+m\frac{H}{2} \phi. \nonumber 
 \end{align}
 Hence, we have 
 \begin{align}
 D^2\phi &= \sum_{i=1}^m  e_i\ccdot_M \nabla_{e_i}\!\left( m \alpha \nu\ccdot \phi+m\frac{H}{2} \phi\right) \nonumber \\
 &= m \alpha\sum_{i=1}^m e_i\ccdot_M {\nabla}_{e_i} \!\left( \nu\ccdot \phi\right) +\frac{m}{2} d H\ccdot_M\phi +\frac{mH}{2}D\phi \nonumber\\
 &=m^2\alpha^2\varphi+\frac{m}{2}\d H\ccdot_M \varphi+\frac{m^2}{4}H^2\phi.\nonumber %\label{eq:D2}
\end{align}
Then, the Schr\"odinger-Lichnerowicz formula \eqref{eq:Lich} on $M$ implies
\begin{align*}
 m(m-1)\left( \alpha^2+\frac{ H^2}{4}\right)\phi +\frac{m-1}{2} dH\ccdot_M \phi = \frac{\text{scal}}{4} \phi +\frac{\i}{2} \Omega^M\ccdot_M\phi, 
\end{align*}
which proves the first identity of the Lemma. Furthermore, using again Equation \eqref{eq:nn}, we obtain for the $\Spinc$ curvature tensor $\mathcal{R}$ on $M$ and for $i\neq j$:
\begin{align*}
\mathcal{R}_{e_j, e_i} \varphi =&  \nabla_{e_j}\nabla_{e_i}\varphi - \nabla_{e_i}\nabla_{e_i}\varphi\\
 =&2\left( \alpha^2 +\frac{H^2}{4}\right) e_i\ccdot e_j \ccdot  \phi -\frac{dH(e_j)}{2}e_i\ccdot_M \phi  +\frac{dH(e_i)}{2}e_j\ccdot_M \phi.
\end{align*}
Hence this implies that
\begin{align*}
\sum_{j=1}^m e_j\ccdot_M \mathcal{R}_{e_j, e_i} \varphi 
 =&2(m-1)\left( \alpha^2 +\frac{H^2}{4}\right) e_i\ccdot_M \ccdot  \phi -\frac{dH}{2}\ccdot_Me_i\ccdot_M \phi  -\frac{m}{2}dH(e_i) \phi.
\end{align*}
The last identity together with the Ricci identity \eqref{Ricci-identity} on $M$ can be written as:
\begin{align*}
\frac 12 \big(\mathrm{Ric} (X) - \i (X \lrcorner\Omega^M)\big)\ccdot_M \varphi =& -\frac 12 dH\ccdot_M X\ccdot_M \varphi - \frac m2 dH(X) \varphi \\ & + \frac{(m-1)}{2} H^2 X\ccdot_M\varphi +2(m-1) \alpha^2 X\ccdot_M \varphi. \qedhere
\end{align*}\end{proof}

%%%%%%%%%%%%%%%%%%%%%%%%%%%%%%%%%%%%%%%%%%%%%%%%%%%%%%%%%%%%%%%%%
 \section{Differential forms on the totally umbilical  hypersurface \texorpdfstring{$M$}{M} build from the  \texorpdfstring{$\alpha$}{alpha}-Killing spinor} 
%%%%%%%%%%%%%%%%%%%%%%%%%%%%%%%%%%%%%%%%%%%%%%%%%%%%%%%%%%%%%%%%%
In this section, we consider again that $(M, g)$ is a totally  umbilical  oriented hypersurface of $(\widetilde M, \widetilde g)$ carrying an $\alpha$-Killing spinor $\psi$. 
 
\begin{lemma}\label{lem_xi}  Let $\xi$ be the vector field on $M$ defined by $g(X, \xi)= -\i \< X\ccdot_M \phi, \phi\>$.  Then on $M$ we have  \begin{align}\label{eq:xi} \xi\lrcorner \Omega^M =& (m-1) |\varphi|^2 dH,\\
\label{eq:dH}
dH(\xi ) =& 0.
\end{align}
\end{lemma}

\begin{proof} We recall that the Killing constant $\alpha$ for $m\geq 2$ is either purely real or purely imaginary. Thus, for $m\geq 1$ we have $(m-1)\alpha^2\in \mathbb R$. Then the real part of the scalar product with $\varphi$ of the Ricci identity \eqref{Ricci-hyper} together with \eqref{cli-diff} gives
$$-\frac{\i}{2} (X\lrcorner \Omega^M \ccdot_M \varphi, \varphi) =\left(\frac{1}{2}dH(X) - \frac m2 dH(X)\right) |\varphi|^2 = \frac{(1-m)}{2} dH(X) |\varphi|^2.$$
Since $(\xi\lrcorner \Omega^M)(X)=-(X\lrcorner \Omega^M)(\xi)= -g(X\lrcorner \Omega^M, \xi)$, we obtain \eqref{eq:xi} and hence \eqref{eq:dH}.
\end{proof}

Next we define differential forms on $M$ depending on whether $\alpha$ is real or imaginary. The first of these forms has been introduced in \cite{herzlich-moroianu}.

\begin{lemma}
On the totally umbilical  hypersurface $M$ of $\widetilde M$, we define differential $p$-forms $\omega_p$ by  
\begin{align*}
\omega_p(e_1, e_2, \ldots, e_p) \define \<(e_1\wedge e_2\wedge\ldots\wedge e_p)\ccdot_M\varphi, \varphi\>. \end{align*} If the Killing constant  $\alpha$ is \emph{real}, we have for all $p\geq 1$,
\begin{align}\label{d-omega}
& d\omega_p = \tfrac{1}{2} H (1-(-1)^p) \omega_{p+1},\\
\label{dH-real}
& dH \wedge \omega_{2p} =0.
\end{align}
\end{lemma}

\begin{proof} Using $[e_i,e_j]=0$, Equations \eqref{eq_spincgaussKill} and \eqref{cli-diff}, we have 
\begin{align*}
 (p+1)& d\omega_p (e_1, e_2, \ldots, e_p, e_{p+1}) \\ 
 =&
  \sum_{j=1}^{p+1} (-1)^{j-1} e_j\big(\omega_p(e_1, e_2, \ldots, \hat e_j, \ldots, e_p, e_{p+1})\big)  \\ 
=& \sum_{j=1}^{p+1} (-1)^{j-1} \<e_1\ccdot_M  \ldots \ccdot_M\hat e_j\ccdot_M \ldots \ccdot_M  e_{p+1}\ccdot_M \left( \alpha e_j \ccdot \varphi - \tfrac{1}{2} H e_j\ccdot_M \varphi\right), \varphi\> \\ 
& + \sum_{j=1}^{p+1} (-1)^{j-1} \<e_1\ccdot_M  \ldots \ccdot_M\hat e_j\ccdot_M \ldots \ccdot_M  e_{p+1}\ccdot_M \varphi, \alpha e_j \ccdot \varphi - \tfrac{1}{2} H e_j\ccdot_M \varphi\> \\
=& \alpha \sum_{j=1}^{p+1} (-1)^{j-1} \left(\<e_1\ccdot \nu\ccdot  \ldots \ccdot \hat e_j \ccdot\hat\nu  \ldots \ccdot  e_{p+1}\ccdot \nu  \ccdot e_j \ccdot\varphi, \varphi\> \right.  \\ 
& \left.+ 
\<e_1\ccdot \nu\ccdot  \ldots \ccdot \hat e_j \ccdot\hat\nu  \ldots \ccdot  e_{p+1}\ccdot \nu  \ccdot\varphi, e_j\ccdot \varphi\>\right)\\ 
& 
-\tfrac{1}{2} H  \sum_{j=1}^{p+1} (-1)^{j-1} \left(\<e_1\ccdot_M \ldots \ccdot_M\hat e_j\ccdot_M \ldots \ccdot_M     e_{p+1}\ccdot_M  e_j \ccdot_M \varphi, \varphi\>\right. \\ &
\left. +\<e_1\ccdot_M  \ldots \ccdot_M\hat e_j\ccdot_M \ldots \ccdot_M     e_{p+1}\ccdot_M  \varphi, e_j\ccdot_M \varphi\>\right) \\ 
%=&
% -\tfrac{1}{2} H  \sum_{j=1}^{p+1} (-1)^{j-1} (-1)^{p+1-i} <e_1\ccdot_M e_2\ccdot_M \ldots \ccdot_M e_j\ccdot_M \ldots \ccdot_M e_p\ccdot_M    e_{p+1}\ccdot_M  \varphi, \varphi> \\ &
% + \tfrac{1}{2} H  \sum_{j=1}^{p+1} (-1)^{j-1} (-1)^{j-1} <e_1\ccdot_M e_2\ccdot_M \ldots \ccdot_M\ e_j\ccdot_M \ldots \ccdot_M e_p\ccdot_M    e_{p+1}\ccdot_M  \varphi, \varphi> \\ 
=&
-(p+1)\tfrac{1}{2} H \Big((-1)^p -1\Big) w_{p+1} (e_1, e_2, \ldots, e_p, e_{p+1}).
\end{align*}
This proves \eqref{d-omega}. In particular,  we obtained for all $k\geq 1$
\begin{equation*}
 \left\{
\begin{array}{rcl}
d\omega_{2k}&=&0,\\
d\omega_{2k-1}&=&H \omega_{2k}.
\end{array}\right.
\end{equation*}
Differentiating the last equality  we obtain $dH \wedge \omega_{2k}  =0$ for any $k \geq 1$.
\end{proof}

\begin{lemma}
On the totally umbilical oriented hypersurface $M$ of $\widetilde M$, we define differential $p$-forms $\eta_p$ by  
\begin{align*}\eta_p(e_1, e_2, \ldots, e_p) \define \<(e_1\wedge e_2\wedge \ldots\wedge e_p)\ccdot_M \varphi, \nu \ccdot\varphi\>.\end{align*}
If the Killing constant $\alpha$ is in $\i\mathbb R\setminus \{0\}$, we have  for $p\geq 1$,
\begin{align}
& d\eta_p = -\tfrac{1}{2} H (1 +(-1)^p)\eta_{p+1},\nonumber\\
& dH \wedge \eta_{2p-1} = 0,  \label{dH-img}
\end{align}
\end{lemma}
\begin{proof}
With an analog calculation as in the last lemma and using $\alpha\in \i \mathbb R$  we obtain
\begin{align*}
(p+1) & d\eta_p (e_1, e_2, \ldots, e_p, e_{p+1}) \\ 
=&\sum_{j=1}^{p+1} (-1)^{j-1} \<e_1\ccdot_M  \ldots \ccdot_M\hat e_j\ccdot_M \ldots \ccdot_M e_{p+1}\ccdot_M \big( \alpha e_j \ccdot \varphi - \tfrac{1}{2} H e_j\ccdot_M \varphi\big), \nu\ccdot\varphi\> \\ & +\sum_{j=1}^{p+1} (-1)^{j-1} \<e_1\ccdot_M \ldots \ccdot_M\hat e_j\ccdot_M \ldots \ccdot_M  e_{p+1}\ccdot_M \varphi, \nu\ccdot \widetilde \nabla_{e_j}  \varphi - \tfrac{1}{2} H e_j\ccdot_M\nu\ccdot \varphi\> \\
=& 
 -\tfrac{1}{2} H  \sum_{j=1}^{p+1} (-1)^{j-1} \<e_1\ccdot_M  \ldots \ccdot_M\hat e_j\ccdot_M \ldots \ccdot_M   e_{p+1}\ccdot_M  e_j \ccdot_M \varphi, \nu\ccdot \varphi\> \\ & 
 -\tfrac{1}{2} H  \sum_{j=1}^{p+1} (-1)^{j-1} \<e_1\ccdot_M \ldots \ccdot_M\hat e_j\ccdot_M \ldots \ccdot_M  e_{p+1}\ccdot_M  \varphi, e_j\ccdot_M\nu\ccdot \varphi\> \\
%  =&
%  -\tfrac{1}{2} H  \sum_{j=1}^{p+1} (-1)^{j-1} (-1)^{p+1-j} \<e_1\ccdot_M  \ldots \ccdot_M e_j\ccdot_M \ldots \ccdot_M  e_{p+1}\ccdot_M  \varphi, \nu\ccdot \varphi\> \\ & 
%  - \tfrac{1}{2} H  \sum_{j=1}^{p+1} (-1)^{j-1} (-1)^{j-1} \<e_1\ccdot_M  \ldots \ccdot_M\ e_j\ccdot_M \ldots \ccdot_M  e_{p+1}\ccdot_M  \varphi, \nu  \ccdot\varphi\> \\ 
=&-(p+1)\tfrac{1}{2} H \Big((-1)^p +1\Big) \eta_{p+1} (e_1, e_2, \ldots, e_p, e_{p+1}),
\end{align*}
and, thus, for all $p\geq 1$
\begin{equation*}
 \left\{
\begin{array}{rl}
d\eta_{2p-1}=&0,\\
d\eta_{2p}=&-H \eta_{2p-1}.
\end{array}\right.
\end{equation*}
Differentiating  the last equality then again gives  $dH \wedge \eta_{2p-1}  =0$ for any $p \geq 1$. 
\end{proof}

%%%%%%%%%%%%%%%%%%%%%%%%%%%%%%%%%%%%%%%%%%%%%%%%%%%%%%%%%%%%%%%%%%%%
\section{Proof of the main result: Theorem ~\ref{Exto}}\label{Mmain}
%%%%%%%%%%%%%%%%%%%%%%%%%%%%%%%%%%%%%%%%%%%%%%%%%%%%%%%%%%%%%%%%%%%%
The goal of this section is to prove Theorem~\ref{Exto}. If $\widetilde{M}$ is spin, $\Omega^M=0$ and the statement follows directly from \eqref{eq:dH}. For the general $\Spinc$ case we split the proof into the two cases:  

{\it Case 1}: The $\alpha$-Killing spinor $\psi$ is  a real Killing spinor ($\alpha \in \mathbb R$)

{\it Case 2}: The $\alpha$-Killing spinor $\psi$ is an  imaginary Killing spinor ($\alpha \in i \mathbb R \setminus\{0\}$).\medskip

First we note, that the hypersurfaces in Section~\ref{Mmain} is not assumed to be orientable. But since all our calculations are local, we at least have locally always an induced $\Spinc$ structure as in Section~\ref{prelim} and can use all the spinorial formula from above.

\begin{proof}[\underline{Proof of Theorem~\ref{Exto} for Case 1}]
We prove this by contradiction. In fact, assume that $dH$ is not identically zero. Then, there is a point $x\in M$ and a neighborhood $U$ of $x$ where $\text{grad}_g H$ is nonzero. Hence, we find a local orthonormal frame $(e_1,\ldots, e_{m-1}, Z= \frac{\text{grad}_g H}{|\text{grad}_g H|})$ of $TU$. Then,  
$(e_1,\ldots, e_{m-1}, Z, \nu)$ is a local orthonormal frame of $\widetilde M$ on $U$. Note that then $dH\cdot_M = \text{grad}_g \cdot_M$.\medskip 

% Let $\partial_H$ be a vector field on $M$ such that $dH(\partial_H)=1$ and $dH(X)=0$ for all $X\in \partial_H^\perp$. Assume  that in a neighborhood $U\subset M$ of some $x \in M$, we have $dH\neq 0$. Let $\{e_1, e_2,\ldots, e_{m-1}\}$ be an orthonormal frame spanning $\partial_H^{\perp}$ in $TU$. We set $e_m\define \frac{\partial_H}{|\partial_H|}$.\\
 
 First we prove the claim for $m>4$:  From Equation (\ref{dH-real}), it is clear that for $2k \leq m-1$ and for each subset $i_1, \ldots, i_{2k}$ of $\{1, \ldots, m-1\}$, we have 
$$\omega_{2k} (e_{i_1}, \ldots, e_{i_{2k}}) = 0.$$
Thus the spinors in  $\{\varphi\}\cup \{ e_{i_1}\ccdot_M e_{i_2} \ccdot_M \varphi\}_{i_1<i_2}\cup ..\cup \{{e_{i_1}\ccdot_M\ldots\cdot_Me_{i_{2l}}\ccdot_M} \varphi\}_{i_1<\ldots <i_{2l}}$, where the $i_j\in \{1, \ldots, m-1\}$, $l= [\frac{m-1}{2}]$ and $\phi=\psi|_x\in \Sigma_x\widetilde{M}$, are mutually orthogonal. Hence they span a complex vector subspace of $\Sigma_x \widetilde{M}$ of complex dimension 
$$\binom{m-1}{0} + \binom{m-1}{2} + \ldots +\binom{m-1}{2l}=2^{m-2}.$$
Since $\mathrm{dim} (\Sigma_x M) = 2^{[\frac m2]}$, we obtain
$ 2^{[\frac m2]} \geq 2^{m-2}$.
It follows that $[\frac m2] \geq m-2$, so $m \leq 4$, which is a contradiction and finishes the proof for $m>4$.\medskip

Let now $m=4$. In dimension $4$  the spinor bundle splits into positive and negative spinors $\Sigma \widetilde{M}|_M\cong \Sigma M\cong \Sigma^+M\oplus \Sigma^-M$, both $\Sigma^{\pm}M$ have $\mathbb C^ 2$-fibers, and we have $\phi\define \psi|_M=\phi^++\phi^-$ with $\phi_{\pm}\in \Gamma(\Sigma^\pm M)$. Moreover,  $e_i\cdot_M\colon \Gamma(\Sigma^\pm M)\to \Gamma (\Sigma^\mp M)$ and   $\overline {\varphi}=- e_1\ccdot_M e_2\ccdot_M e_3\ccdot_M Z\ccdot_M \varphi$.\medskip 

Using Equation \eqref{dH-real} we have
\begin{align*}
 0=&(dH\wedge \omega_2)(Z, e_2,e_3)= dH(Z) \omega_2(e_2,e_3)= dH(Z) \< e_2\cdot_M e_3\cdot_M \varphi, {\varphi}\>\\ =& \< Z\cdot_M e_2 \cdot_M e_3\cdot_M \varphi , dH\cdot_M\varphi\>\\
 = &- \<e_1\cdot_Me_1\cdot_M  e_2 \cdot_M e_3\cdot_M Z\cdot_M \varphi , dH\cdot_M\varphi\>
 =\< \overline{\varphi}, dH\cdot_Me_1\cdot_M\varphi\>
\end{align*}
and analogously $\< dH\cdot_Me_i\cdot_M\varphi, \overline{\varphi}\>=0$ for $i=1,2,3$.\medskip 

Let $\xi$ be as defined in Lemma~\ref{lem_xi}. Then \eqref{eq:dH} implies that $\xi$ is in the span of $\{e_1,e_2,e_3\}$. Taking the Clifford multiplication with $dH\ccdot_M$ in the Ricci identity  \eqref{Ricci-hyper} for $X = \xi$ and then the imaginary part of the scalar multiplication with $\overline\varphi$, we obtain
$$0=\<\xi\lrcorner \Omega^M,Z\>\<\varphi,\overline{\varphi}\>.$$
Together with \eqref{eq:xi} this implies  $\vert dH\vert^2 |\varphi|^2 \<\varphi,\overline{\varphi}\>=0$. Since $\varphi\neq 0$ for a real Killing spinor and $dH|_U \neq 0$ by assumption, we obtain $\<\varphi,\overline{\varphi}\>=0$.  Let $X \in \Gamma(TM)$ with $|X|=1$. Using $\nu\cdot \colon \Gamma(\Sigma^\pm M)\to \Gamma(\Sigma^\pm M)$, see \cite[p. 31]{ginoux-these},  we calculate $\widetilde\nabla_X \overline \varphi = -\alpha X\ccdot\overline \varphi $. Differentiating  $\<\varphi,\overline{\varphi}\>=0$ and using Equation \eqref{eq_spincgauss} we then obtain
$$\alpha \<X\cdot \varphi, \overline \varphi\> - \tfrac{1}{2} H \<X\cdot_M\varphi, \overline \varphi\> + \tfrac{1}{2} H \<\varphi, X\cdot_M\overline\varphi\> - \alpha \<\varphi, X\cdot \overline \varphi\> =0.$$
Hence, we have 
\begin{align}
2\alpha  \<X\cdot \varphi, \overline \varphi\> =& H  \<X\cdot_M\varphi, \overline \varphi\>.\label{derivative2}
\end{align}
Let also $e_4\define Z$. We calculate using $\nabla_{e_j} e_i=0$ (and hence $\widetilde{\nabla}_{e_j} e_i = H \delta_{ij} \nu$) that 
\begin{align}
 e_j\<e_i\cdot \phi, \overline\phi\>&= \< \widetilde \nabla_{e_j} (e_i\cdot \phi),\overline{\phi}\> +\<e_i\cdot \phi, \widetilde\nabla_{e_j} \overline{\phi}\>\nonumber\\
 &= \< H\delta_{ij} \nu\cdot \phi +  e_i\cdot \widetilde{\nabla}_{e_j} \phi,\overline{\phi}\> +\<e_i\cdot \phi, -\alpha e_j\cdot \overline{\phi}\>\nonumber\\ 
 &= \< H\delta_{ij} \nu\cdot \phi, \overline{\phi}\>+ \< \alpha e_i \cdot e_j\cdot \phi,\overline{\phi}\> -\<e_i\cdot \phi, \alpha e_j\cdot \overline\phi\>\nonumber\\ 
 &= H\delta_{ij} \<  \nu\cdot \phi, \overline{\phi}\>\label{eq:aux1}
\end{align}
and
\begin{align}
 e_j&\<e_i\cdot_M \phi, \overline\phi\>= \< e_i\cdot_M \nabla_{e_j} \phi,\overline{\phi}\> +\<e_i\cdot_M \phi, \nabla_{e_j} \overline{\phi}\>\nonumber\\
 &= \< \alpha e_i\cdot_M \cdot e_j\cdot \phi\!-\!\tfrac{1}{2} H e_i\cdot_M e_j\cdot_M \phi,\overline{\phi}\> \!+\!\<e_i\cdot_M \phi, -\alpha e_j\cdot \overline\phi-\tfrac{1}{2} H e_j\cdot_M \overline{\phi}\>\nonumber\\ 
 &=-H \< e_i \cdot_M e_j\cdot_M \phi, \overline{\phi}\>.\label{eq:aux2}
\end{align}
By \eqref{cli-diff} and $\overline{\phi}=-e_1\ccdot_M e_2\ccdot_Me_3\ccdot_MZ\ccdot_M \phi$,  the left hand sides of both of the equations \eqref{eq:aux1} and \eqref{eq:aux2} are real. On the other hand the right hand side of \eqref{eq:aux1} is imaginary and the one of \eqref{eq:aux2} is imaginary for $i\neq j$ and $0$ for $i=j$. Hence, all sides have to be zero. Using this when differentiating \eqref{derivative2} for $X=e_i$ in direction of $Z$, we obtain 
\[ Z(H) \<e_i\cdot_M \phi, \overline{\phi}\>=0 \text{ for all }i=1,\ldots, 4.\]
Hence, $\<e_i\cdot_M \phi, \overline{\phi}\>=0$ and thus
\begin{equation}\label{eq:re4} \Re\<e_i\cdot_M \phi_+, \phi_-\>=0  \text{ for all }i=1,\ldots, 4.\end{equation}
We note that in dimension $4$ every non-zero element $\psi\in \Sigma_+M|_y$ for $y\in U$ gives rise to a real basis $e_i\cdot_M \psi$ of $\Sigma_+M|_y$ with respect to the scalar product $(.,.)\define \Re \< .,.\>$. 
Hence, Equation \eqref{eq:re4} implies that at each point $y\in U$, either $\varphi_+ = 0$ or $\varphi_-$ is perpendicular to the four dimensional real vector space $\Sigma_+M|_y$ w.r.t this real scalar product, i.e, $\varphi_- = 0$.\medskip

Since $|\varphi|^2=|\varphi_+|^2+|\varphi_-|^2$ is of constant norm  by \eqref{eq:realconst}, we obtain that $\varphi_+ = 0$ or $\varphi_- =0$ on all of $U$. Assume that $\varphi_- = 0$ on $U$ (the other case is analogous), then 
 $$0=\nabla_X \varphi_- = \alpha X\ccdot \varphi_+ + \tfrac{1}{2} H X\ccdot_M \varphi_+.$$ 
The real part of the scalar product of the last identity with $X\ccdot_M \varphi_+$ gives
$$ \tfrac{1}{2} H |X|^2 \vert\varphi_+\vert^2 = 0.$$
Since $\varphi$ is  non-zero, $\varphi_+$ has no zeros on $U$ and we get that $H=0$ on $U$.  Thus, $dH = 0$ on $U$ which gives the contradiction.
\end{proof}

\begin{proof}[\underline{Proof of Theorem~\ref{Exto} for Case 2}]  Assume that $dH$ is not identically zero. Then, there is a point $x\in M$ and a neighborhood $U$ of $x$ where $\text{grad}_g H$ is nonzero. Hence, we find a local orthonormal frame $(e_1,\ldots, e_{m-1}, Z= \frac{\text{grad}_g H}{|\text{grad}_g H|})$ of $TU$. Then, we have with $(e_1,\ldots, e_{m-1}, Z, \nu)$ again a local orthonormal frame of $\widetilde M$ on $U$. \medskip 

On all of $U$ we have by Equation \eqref{dH-img} that $d H \wedge \eta_1 = 0$. Then with $dH(e_i)=0$ we obtain
$$0 = d H \wedge \eta_1 \left(\frac{\text{grad}_g H}{|\text{grad}_g H|^2}, e_i\right) = \eta_1 (e_i) = -\<e_i \ccdot\varphi, \varphi\>$$ for all $1\leq i\leq m-1$ which will used in following without any further comment. \medskip

\underline {We consider three different subcases}: First assume that $\<dH\cdot \varphi, \varphi\>=\<\nu \cdot \varphi, \varphi\>=0$ on $U$. Note that for all $X\in \Gamma (TM)$ the vector $V$, defined on $\widetilde{M}$ by $\widetilde{g}(V,X) \define \i \< X\ccdot \varphi, \varphi\>$ , vanishes on $U$, see \cite{rad, baum1, baum2, baum3, grosse-nakad}. From \cite{baum1, rad} we have $\widetilde{\nabla}_X V=2\alpha |\varphi|^2 X$ for all $X\in \Gamma(T\widetilde{M})$. Since $V\equiv 0$, this implies that $\varphi\equiv 0$ on $U$.
This gives a contradiction in the first case.\medskip 

Second let $\<dH\cdot \varphi, \varphi\>=0$ and let $\<\nu \cdot \varphi, \varphi\>$ be nonzero on a possibly smaller $U$. In particular, we can make $U$ small enough such that $\phi$ has no zeros on $U$. Then, 
the   imaginary part  of the  scalar product  of Equation~\eqref{riccianddH} with $\nu\ccdot\varphi$  gives
\begin{align*}\widetilde{\text{Ric}}(\nu, \nu) = 4m\alpha^2.\end{align*}
Reinserting into Equation \eqref{riccianddH} gives
\begin{align}\label{contr}
\frac{m-1}{2} dH\ccdot_M\varphi =& \frac{\i}{2} (\nu\lrcorner\widetilde \Omega)\ccdot_M \phi,
\end{align}
and hence 
\begin{align*}
\frac{m-1}{2} \<dH\ccdot_M\varphi, Z\ccdot_M \phi\> =& \frac{\i}{2} \left(\sum_{i=1}^{m-1} \widetilde \Omega(\nu, e_i)\< e_i\ccdot_M \varphi, Z\ccdot_M \phi\> + \widetilde \Omega(\nu, Z)|Z\ccdot_M \varphi|^2\right). 
\end{align*}
Taking the imaginary part of the last equality implies $\widetilde \Omega(\nu, Z)|\phi|^2=0$.  Since $\phi$ has no zeros, we obtain  $\widetilde \Omega (Z, \nu) =0$. The real part of the scalar product of Equation \eqref{eq:LichM} with $e_i\cdot \phi$ then gives
\begin{align*} \frac{m-1}{2}&\< dH\cdot \nu\cdot \phi, e_i\cdot \phi\>\\ = &-\frac{1}{2}\text{Im} \left( \sum_{j<k} \Omega^M(e_j,e_k)\< e_j\ccdot e_k\ccdot \phi, e_i\ccdot \phi\> + \sum_{j} \Omega^M(e_j,Z)\< e_j\ccdot Z \ccdot \phi, e_i\ccdot \phi\> \right) =0, \nonumber
\end{align*}
where the last equality uses Equation \eqref{cli-diff} and $\< Z\cdot \phi,\phi\>=\<e_i\cdot \phi,\phi\>=0$.
Thus, taking the scalar product of Equation \eqref{contr} with $e_i\cdot \phi$ implies 
\begin{align*}
0= \sum_{j=1}^{m-1} \widetilde \Omega(\nu, e_j)\< e_j\ccdot_M \varphi, e_i\cdot  \phi\>. 
\end{align*}
By taking the imaginary part we obtain $\widetilde \Omega (\nu, e_i)=0$ and hence $\nu\lrcorner \widetilde{\Omega}=0$. Reinserting in Equation \eqref{contr} implies $dH=0$ which gives the contradiction in the second case.\medskip 

The third case covers the remaining possibility that $\< dH\cdot \varphi, \varphi\>$ is nonzero at a point in $U$. The next calculations will be carried out at this point. Taking the real part of the scalar product of Equation \eqref{eq:LichM} with $e_i\cdot \phi$ gives
\begin{align}\label{eq:iHnu0} \frac{m-1}{2}&\< dH\cdot \nu\cdot \phi, e_i\cdot \phi\> = \frac{\i}{2} {\Omega^M} (e_i, Z)\< Z\cdot \phi, \phi\>.
\end{align}
On the other hand taking the scalar product of the Ricci identity \eqref{Ricci-hyper}  for $X=e_i$ with $\nu\cdot\varphi$  gives 
 \begin{align*}
 - \frac{1}{2}& \Ric\left(e_i,Z\right) \<Z\cdot \varphi, \varphi\> + \frac{\i}{2} \Omega^M\left(e_i, Z\right) \< Z\cdot \varphi, \varphi\>
= -\frac 12 \< dH \cdot \nu \cdot \varphi,  e_i \cdot\varphi\>.
\end{align*}

The imaginary part of the last identity gives $\Ric\left(e_i,Z\right)=0$. Reinserting this into the above equation and using Equation \eqref{eq:iHnu0}  implies $\< dH \cdot \nu \cdot \varphi,  e_i \cdot\varphi\>=\Omega^M\left(e_i, Z\right)=0$.\medskip

Taking again the scalar product of the Ricci identity \eqref{Ricci-hyper} but this time for $X=Z$ then 
gives 
 \begin{align*}
  \frac{1}{2} &\Ric\left(Z ,Z\right)Z \cdot_M \varphi  = \frac{1-m}{2} |dH| \varphi + \left(\frac 12 (m-1)H^2 - 2(m-1) \alpha^2\right) Z\cdot_M\varphi.
\end{align*}
The real part of the scalar product with of the last identity with $\varphi$ gives that $\frac{1-m}{2} | dH| \vert \varphi\vert^2= 0$. But  $\varphi$ has no zeros on $U$ and $m>1$, so $dH=0$ which gives the desired contradiction for the remaining case.
\end{proof}

The following example shows that the dimension constraint  for Case 1 in Theorem~\ref{Exto} is necessary. 

\begin{example}There exist totally umbilical connected  hypersurfaces with non-constant mean curvature in Riemannian $\Spinc$ manifolds of dimension $3$ or $4$ and  carrying parallel or real Killing spinors:

\textit{\underline{Dimension $3$:}} The product of the canonical $\Spinc$ structure on $\mathbb S^2$ with the $\Spin$ structure on $\mathbb R$ defines a $\Spinc$ structure on the manifold $\widetilde M = \mathbb S^2 \times \mathbb R$ \cite {moroi}. This $\Spinc$ structure carries a parallel spinor \cite{moroi}. Totally umbilical  hypersurfaces (which are not totally geodesic) of $\mathbb S^2 \times \mathbb R$ have been classified in \cite{Souam1}. Moreover, they are not of constant mean curvature \cite[ Remark 10]{Souam1}.  We point out that  $\widetilde M = \mathbb S^2 \times \mathbb R$ is $\Spin$ but does not carry a real or parallel Killing $\Spin$ spinor.

\textit{\underline{Dimension $4$:}} The $\Spinc$ manifold $\widetilde M = \mathbb S^2 \times \mathbb H^2$  carries a parallel spinor for the product of the canonical $\Spinc$ structure on $\mathbb S^2$ with the canonical $\Spinc$ structure on $\mathbb H^2$. In \cite{danielthesis}, the author classified totally umbilical  hypersurfaces of $\mathbb S^2 \times \mathbb H^2$  (see \cite[Theorem 4.5.3]{danielthesis}) and showed that these hypersurfaces are not of constant mean curvature in general. We also point out that  $\widetilde M = \mathbb S^2 \times \mathbb H^2$ is $\Spin$ but does not carry a real or parallel Killing $\Spin$ spinor. 
\end{example}

\section{Extrinsic hyperspheres in Riemannian \texorpdfstring{$\Spin$}{Spin} manifolds}

In this section, we give some additional information if the ambient manifold carrying a Killing spinor is already spin. As a first corollary, we get:
\begin{cor}\label{cor-E}
Let $M^m \hookrightarrow \widetilde M^{m+1}$ be a totally umbilical   isometric immersion. Assume that $\widetilde M$ is a $\spin$ manifold with a Killing spinor $\psi$ of Killing constant $\alpha$ (could be zero, real or purely imaginary). Then, $M$ is Einstein with scalar curvature $\scal = m(m-1)(H^2 + 4 \alpha^2)$.  
\end{cor}
\begin{proof}[Proof of Corollary~\ref{cor-E}]
From Thm \ref{Exto}, $H$ is constant. By the Ricci identity \eqref{Ricci-hyper}, we have $$\frac 12 \mathrm{Ric}(e_j) \ccdot_M \phi = 2 (m-1) \left[\alpha^2 + \frac {H^2}{4}\right]e_j \ccdot_M \phi.$$  This means that $M$ is Einstein with constant scalar curvature $\scal = m(m-1)(H^2 + 4 \alpha^2)$.  
\end{proof}
For later use, we recall here \textbf{Koiso's Theorem.} 
\begin{thm}\cite[Thm. B]{koiso} \label{koi}Let $M$ be a totally umbilical  Einstein hypersurface in a complete Einstein manifold $(\widetilde M, \widetilde g)$. Then the only possible cases are: 
\begin{enumerate}
\item $g$ has  positive Ricci curvature. Then $g$ and $\widetilde g$ have constant sectional curvature. 
\item $\widetilde g$ has negative Ricci curvature. If $\widetilde M$ is compact or homogeneous, then $g$ and $\widetilde g$ have constant sectional curvature.
\item $g$ and $\widetilde g$ have zero Ricci curvature. If $\widetilde M$ is simply connected, then $\widetilde M = (\overline M, \overline g) \times \mathbb R$ where $\overline M$ is totally geodesic hypersurface in $\widetilde M$ which contains $M.$ 
\end{enumerate}
\end{thm}

An important special case of totally umbilical  hypersurfaces with constant mean curvature are totally geodesic hypersurfaces (when the mean curvature $H$ is  zero). The other cases are called extrinsic hypersphere (when the mean curvature is a nonzero constant).\medskip 

From Theorem~\ref{koi} and Corollary \ref{cor-E}, we deduce the following result:
\begin{thm}\label{pa-ap}
Let $\widetilde M$ be a complete Riemannian $\Spin$ manifolds of non-constant sectional curvature  that carry an $\alpha$-Killing spinor. 
If $\alpha \in \i\mathbb R\setminus\{0\}$, we assume moreover that $\widetilde M$ is homogeneous. Then, there are no extrinsic hyperspheres in $\widetilde M$.
 \end{thm}
\begin{proof} Assume that $M$ is an extrinsic hypersphere ($H\neq 0$) in a Riemannian $\Spin$ manifold with an $\alpha$-Killing spinor. By Corollary \ref{cor-E}, $M$ is  Einstein with scalar curvature $m(m-1)(H^2 + 4 \alpha^2)$.  If  $H^2 + 4 \alpha^2 > 0$, the Ricci curvature of $M$ is positive.  If  $H^2 + 4 \alpha^2\leq 0$, then $\alpha\in \i\mathbb R\setminus\{0\}$ and hence  the Ricci curvature of $\widetilde{M}$ is negative  and hence, in both cases, we have by Koiso's theorem \ref{koi} that $\widetilde g$ is of constant sectional curvature, which is a contradiction. 
\end{proof}
Theorem~\ref{THMp1} is a particular case of  Theorem~\ref{pa-ap} as is easily seen as follows: All the manifolds appearing in this Theorem are  $\Spin$, complete, with $\alpha$-Killing spinor and  of non-constant sectional curvature (see \cite{boyer} and \cite[Prop~3.1]{Gou}).\medskip 

One can add further examples for Theorem~\ref{pa-ap}, such as $6$-dimensional nearly K\"ahler manifolds which are not K\"ahler and of non-constant sectional curvature and  $7$-dimensional weak $G_2$ manifolds of non-constant sectional curvature.\medskip 

The completeness assumptions in Theorem~\ref{pa-ap} is necessary not only because we want to use Koiso's theorem but also because otherwise, every
manifold is an extrinsic hypersphere in its (non-complete) metric cone.


\begin{thebibliography}{99}
\bibitem{bar1} C. B\"ar, \textit{Real Killing spinors and holonomy}. Comm. Math. Phys. 154, 3 (1993), 509-521.

\bibitem{Bar98} C. B\"ar, \textit{Extrinsic bounds for eigenvalues of the Dirac operator}, Ann. Glob. Anal. Geom. 16 (1998), 573-596.

\bibitem{baum1} H. Baum, \textit{Complete Riemannian manifolds with imaginary Killing spinors},  Ann. Global Anal. Geom. 7, 3
(1989), 205-226.
\bibitem{baum2} H. Baum, \textit{Odd-dimensional Riemannian manifolds with imaginary Killing spinors}, Ann. Global Anal.
Geom. 7, 2 (1989), 141-153.
\bibitem{baum3} H. Baum, \textit{Vari\'et\'es riemanniennes admettant des spineurs de Killing imaginaires}, C. R. Acad. Sci. Paris S\'er.I Math. 309, 1 (1989), 47–49.

\bibitem{bfg}  H. Baum,  T. Friedrich, R. Grunewald and I. Kath, \textit{ Twistor and Killing spinors on Riemannian manifolds}, vol. 108 of Seminarberichte [Seminar Reports]. Humboldt Universit\"at Sektion Mathematik, Berlin, 1990.

\bibitem{BCO} J. Berndt, S. Console, C. Olmos, \textit{Submanifolds and Holonomy}, vol. 434. Chapman and  Hall, Boca Raton (2003).
\bibitem{chen1} B. Y. Chen, \textit{Extrinsic spheres in compact symmetric spaces are intrinsic spheres},
Michigan Math. J. 24 (1977), no. 3, 265-271.
\bibitem{chen2} B. Y. Chen, \textit{Odd-dimensional extrinsic spheres in K\"ahler manifolds}, 
Rend. Mat. (6) 12 (1979), no. 2, 201-207.
\bibitem{chen3} B. Y. Chen, \textit{Classification of totally umbilical  submanifolds in symmetric spaces}.
J. Austral. Math. Soc. (Series A) 30 (1980), 129-136.
\bibitem{chen4} B. Y. Chen, T. Nagano, \textit{Totally geodesic submanifolds of symmetric spaces,  II},
Duke Math. J. 45 (1978), no. 2, 405-425.
\bibitem{Chen1} B.~Chen, \textit{Totally umbilical  submanifolds of K\"ahler manifolds},  Arch. Math., Vol. 36, 83--91 (1981).

\bibitem{Gou} M. Bertola and D. Gouthier, \textit{Warped products with special Riemannian curvature}, Bol. Soc. Bras. Mat., Vol 32, No. 1, 2001, 45-62.

\bibitem{BHMM} J. P. Bourguignon, O. Hijazi, J. L. Milhorat and A. Moroianu, \textit{A spinorial approach to
Riemannian and conformal geometry}, EMS Monographs in Mathematics, 2015.
\bibitem{boyer} C. Boyer and K. Galicki,  \textit{3-Sasakian manifolds}, Surveys Diff. Geom. 7 (1999), 123.

\bibitem{frfr}T. Friedrich, \textit{Der erste Eigenwert des Dirac-operators einer kompakten Riemannschen
Mannigfaltigkeit nichtnegativer  Skalarkr{\"u}mmung}, Math. Nach. 97 (1980), 117-146.


\bibitem{friedrich} T. Friedrich, \textit{Dirac operators in Riemannian Geometry}, Graduate studies in mathematics, Volume 25, Americain Mathematical Society.

\bibitem{ginoux-these} N. Ginoux, \textit{Op\'erateurs de Dirac sur les sous-vari\'et\'es},  Ph. D thesis 2002, Institut \'Elie Cartan, Nancy.

\bibitem{grosse-nakad} N. Gro{\ss}e and R. Nakad, \textit{Complex Generalized Killing Spinors on Riemannian Spin$^c$ manifolds},  Results in Mathematics, Vol. 67, Issue 1 (2015), 177-195.

\bibitem{herzlich-moroianu} M. Herzlich and A. Moroianu, \textit{Generalized Killing spinors and conformal eigenvalue estimates for Spin$^c$ manifolds}, Ann. Glob. Anal. Geom. 17 (1999), 341-370.

\bibitem{HS14} O. Hijazi and S. Montiel, \textit{A holographic principle for the existence of parallel
spinors and an inequality of Shi-Tam type}, Asian J. Math. 18 (2014), no. 3, 489-506.





\bibitem{hij84} O. Hijazi, \textit{Op\'erateurs de Dirac sur les vari\'et\'es riemanniennes: minoration des
valeurs propres}, Ph. D thesis, Ecole Polytechnique, 1984. 


\bibitem{hij86} O. Hijazi, \textit{A conformal Lower bound for the smallest eigenvalue of the Dirac operator
and Killing spinors}, Commun. Math. Phys. 104 (1986), 151-162.

\bibitem{HMR03} O. Hijazi, S. Montiel and A. Roldan,\textit{
Dirac operators on hypersurfaces
of manifolds with negative scalar curvature}, Ann. Global Anal. Geom. 23 (2003), 247-264.


\bibitem{HMU} O. Hijazi, S. Montiel and F. Urbano,
\textit{ Spin$^c$ geometry of K\"ahler manifolds and the Hodge Laplacian on minimal Lagrangian submanifolds}, Math. Z. 253, Number 4 (2006) 821-853.

\bibitem{HMZ01} O. Hijazi, S. Montiel and X. Zhang, \textit{Dirac operator on embedded hypersurfaces}, Math. Res. Lett.
8 (2001), 195-208.



\bibitem{HM13} O. Hijazi and S. Montiel, \textit{A spinorial characterization of hyperspheres}, Calc. Var. Part. Diff. Eq.48 (2013), 527-544.

\bibitem{JMS} T. Jentsch, A. Moroianu and U. Semmelmann, \textit{Extrinsic hypersphere in manifolds with special holonomy}, Differential Geometry and its Applications, 31 (2013) 104-111. 
\bibitem{danielthesis} D. Kowalczyk,  \textit{Submanifolds  of product spaces}, Ph.D. thesis, Katolieke Universiteit LEUVEN, 2011. 
\bibitem{koiso} N. Koiso, \textit{Hypersurfaces of {E}instein manifolds}, 
   Ann. Sci. \'{E}cole Norm. Sup. 14(4) (1981), 433--443.
\bibitem{Kow} O. Kowalski, \textit{Properties of hypersurfaces which are characteristic for spaces of constant curvature}, Annali della Scuola Normale Superiore di Pisa, Classe di Scienze,  Serie 3, Volume 26 (1972) no. 1, p. 233-245.
\bibitem{KR} W. K\"{u}hnel and H.-B Rademacher,  \textit{Twistor Spinors with zeros}, Int. J. Math., 1994, 5,
877-895.



\bibitem{LM} H. B. Lawson and M. L. Michelson, \textit{Spin geometry}, Princeton University press, Princeton,
New Jersey, 1989.

\bibitem{Lich1} A. Lichnerowicz, \textit{Sur les z\'eros des spineurs-twisteurs}, C.R. Acad. Sci. Paris, 1990, 310, 19-22.
\bibitem{moroi} A. Moroianu, \textit{Parallel and Killing spinors on Spin$^c$ manifolds}, Comm. Math. Phys. 187, 2 (1997), 417-427.
\bibitem{r2} R. Nakad, \textit{The Energy-Momentum tensor on Spin$^c$ 
manifolds}, Advances in Mathematical Physics, vol. 2011, Article ID 471810, doi:10.1155/2011/471810.

\bibitem{Na11} R. Nakad, \textit{Sous-vari\'et\'es sp\'eciales des vari\'et\'es spinorielles complexes}, Ph.D. thesis, Universit\'e Henri Poincar\'e-Nancy I, (2011).


\bibitem{ok} M. Okumuara, \textit{Totally umbilical  hypersurfaces of a locally product Riemannian manifolds}, Kodai Math. Sem. Rep., 19 (1967), 35-42.
\bibitem{ON} B. O'Neill, \textit{Semi-Riemannian geometry}, Academic Press, New York, London, 1983. 

\bibitem{rad} H.B. Rademacher, \textit{Generalized Killing spinors with imaginary Killing
function and conformal Killing
fields}, In Global differential geometry and global analysis (Berlin,1990),  vol. 1481 of Lecture Notes in Math. Springer, Berlin, 1991, pp. 192-198
\bibitem{RN12} J. Roth and R. Nakad, \textit{Hypersurfaces of Spin$^c$ manifolds and Lawson type correspondence},  Ann. Glob. Anal. Geom., Vol. 42, No. 3 (2012), 421-442.

\bibitem{Souam1} R. Souam and E. Toubiana, \textit{Totally umbilic  surfaces in homogeneous 3-manifolds}, Comment. Math. Helv. 84 (2009), 673-704

\bibitem{to} K. Tojo, Extrinsic hyperspheres of naturally reductive homogeneous spaces. Tokyo J. Math. 20.
(1997), no. 1, 35-43.
\bibitem{tsu} K. Tsukada, \textit{Totally geodesic hypersurfaces of naturally reductive homogeneous spaces}, Osaka J. Math.
33 (1996), no. 3, 697-707.
\bibitem{yam} S. Yamaguchi, H. Nemoto and N. Kawabata, \textit{Extrinsic spheres in a K\"ahler manifold}.
Michigan Math. J. 31 (1984), no. 1, 15-!9.


\bibitem{wang} M. Wang, \textit{Parallel spinors and parallel forms}, Ann. Glob. Anal. Geom. 7, 59-68 (1989).








\end{thebibliography}
\end{document}